\documentclass[11pt,oneside,a4paper]{article}
\usepackage[utf8]{inputenc}
\usepackage{microtype}
\usepackage{setspace} 

\usepackage[style=apa]{biblatex}
\usepackage{graphicx,xcolor}
\usepackage{siunitx}
\usepackage{amsmath}
\usepackage[ruled,vlined]{algorithm2e}
\usepackage{twoopt}
\usepackage{xifthen}
\usepackage{amsthm,amssymb}
\usepackage{multirow} 
\usepackage{hyperref}
\usepackage{tabulary}
\usepackage{booktabs} 

\theoremstyle{plain}
\newtheorem{thm}{Theorem}[section]

\newtheorem{lem}[thm]{Lemma}

\newtheorem{defn}[thm]{Definition}

\newtheorem{ass}{Assumption}
\newtheorem{prob}{Problem}

\LinesNumbered
\SetKwInOut{Input}{input}
\SetKwInOut{Output}{output}
\SetKw{And}{and}
\SetKw{Not}{not}
\SetKwProg{Fn}{Function}{ is}{end}

\newcommand{\FG}[1][]{ 
    \Gamma\ifthenelse{\isempty{#1}}{}{^{#1}}%
}

\newcommand{\fd}[2]{\mathrm{delay\_plan}(#1,\;#2)}
\newcommand{\ftc}[3]{\mathrm{try\_connect}(#1,\;#2)}

\newcommand{\vertex}{\kappa}
\newcommand{\vertices}{K}
\newcommand{\tor}[1]{t^{\mathrm{o}}_{#1}} 
\newcommand{\tde}[1]{t^{\mathrm{d}}_{#1}} 
\newcommand{\tst}[1]{t^{\mathrm{s}}_{#1}} 
\newcommand{\dep}[1]{\delta_{#1}} 
\newcommand{\dmax}[1]{\delta^{\mathrm{max}}_{#1}} 
\newcommand{\ftt}[2]{f_{\mathrm{tt}}(#1,\;#2)}
\newcommand{\fc}[2]{f_{\mathrm{tc}}(#1,\;#2)}
\newcommand{\PV}[1]{
    \Phi\ifthenelse{\isempty{#1}}{}{_{#1}}
}

\title{Optimal Chaining of Vehicle Plans with Time Windows}

\author{David Fiedler$^1$\thanks{Corresponding author.} \and Fabio V. Difonzo$^2$ \and Jan Mrkos$^1$}

\date{
    \small$^1$Department of Computer Science, Faculty of Electrical Engineering, CTU in Prague, Karlovo náměstí 13, 121 35 Praha 2, Czech Republic \\ \texttt{\{david.fiedler, jan.mrkos\}@agents.fel.cvut.cz}\\%
	$^2$Dipartimento di Matematica, Universit\`a degli Studi di Bari Aldo Moro, Via E. Orabona 4, 70125 Bari, Italy \\ \texttt{fabio.difonzo@uniba.it}\\[2ex]%
    \today
}

\begin{document}

\maketitle

\begin{abstract}
For solving problems from the domain of Mobility-on-Demand (MoD), we often need to connect vehicle plans into plans spanning longer time, a process we call \emph{plan chaining}. 
As we show in this work, chaining of the plans can be used to reduce the size of MoD providers' fleet (fleet-sizing problem) but also to reduce the total driven distance by providing high-quality vehicle dispatching solutions in MoD systems.
Recently, a solution that uses this principle has been proposed to solve the fleet-sizing problem~\cite{vazifehAddressingMinimumFleet2018}.
The method does not consider the time flexibility of the plans. Instead, plans are fixed in time and cannot be delayed. However, time flexibility is an essential property of all vehicle problems with time windows.
This work presents a new plan chaining formulation that considers delays as allowed by the time windows and a solution method for solving it.
Moreover, we prove that the proposed plan chaining method is optimal, and we analyze its complexity.
Finally, we list some practical applications and perform a demonstration for one of them: a new heuristic vehicle dispatching method for solving the static dial-a-ride problem.
The demonstration results show that our proposed method provides a better solution than the two heuristic baselines for the majority of instances that cannot be solved optimally. At the same time, our method does not have the largest computational time requirements compared to the baselines.
Therefore, we conclude that the proposed optimal chaining method provides not only theoretically sound results but is also practically applicable.
\end{abstract}




\section{Introduction}
The fastest-growing mode in urban mobility is on-demand mobility, mostly realized by transportation network companies like Uber or Lyft.
Mobility-on-Demand (MoD) has an advantage over private vehicles in reducing the fleet size (and consequently, the parking space) due to \emph{carsharing}: one car can serve many travel requests during one day.
Moreover, some MoD options (e.g., Uber Pool) allow users to share rides (\emph{ridesharing}), and, as a result, they reduce the total distance driven over traveling separately.

One of the key problems regarding MoD systems is to determine the minimal vehicle fleet able to serve all travel requests: the \emph{fleet sizing} problem.
By reducing the fleet size, we can reduce the capital cost of the system by reducing the number of vehicles and the parking space needed.
Moreover, we can reduce the operational cost by reducing the number of drivers needed.
Another important problem is the \emph{vehicle dispatching}: a problem of assigning vehicles to requests and determining the vehicle plans (routes).
This problem is very complex especially when ridesharing is employed, as the number of possible plans grows exponentially with the number of requests.
In operational research field, this problem is known as the \emph{dial-a-ride problem} (DARP).
By providing high-quality vehicle dispatching solutions, we can reduce the operational cost of the system by reducing the total distance driven by the vehicles.
Moreover, by sharing rides, we can even reduce the required fleet size with all the benefits mentioned above.

In the MoD context, we often need to connect vehicle plans to longer plans.
For example, we can connect one vehicle plan starting at 7:00 and ending at 8:00 to another plan starting at 8:30 and ending at 9:00, resulting in a plan starting at 7:00 and ending at 9:00.
This process is the central theme of this work, we call it \emph{plan chaining}.

Below, in Section~\ref{sec:fleet_sizing_intro}, we will briefly introduce the research on the fleet-sizing problem, highlighting the use of connecting the plans to solve it.
Then, in Section~\ref{sec:ridesharing_intro}, we present the ridesharing research and suggest the possibility of using plan chaining to solve the DARP.
Finally, in Section~\ref{sec:contribution}, we will summarize the contribution of this work.

\subsection{Fleet sizing}
\label{sec:fleet_sizing_intro}
In the well-known 2018 article,~\textcite{vazifehAddressingMinimumFleet2018} solves the fleet sizing problem optimally using two steps. 
First, they generate a shareability network: a graph where the nodes are the plans for solving each request, and edges are the possible connections between those plans.
The second step is to minimize the number of edges in the shareability network.
The authors prove that the shareability network is an acyclic graph, and thus, the problem can be solved in polynomial time by the Hoptcroft-Carp algorithm applied to a bipartite graph corresponding to the shareability network. 

Inspired by this work,~\textcite{wangMinimumFleetRidesharingAware2021} propose a method to solve the fleet-sizing problem in a ridesharing context.
The authors formalize the problem as finding a minimum cover in the set of all possible dispatching graphs: the dispatching tree cover. 
They prove that the dispatching tree cover problem is NP-Hard and propose a heuristic algorithm; then, they evaluate its performance both analytically and on a case study in Shenzhen. 

\textcite{xuNetworkFlowBasedEfficientVehicle2022} extend~\textcite{vazifehAddressingMinimumFleet2018} by considering the limitation of the number of available vehicles in particular zones.
To achieve this, they transform the shareability network into a min-cost flow problem, which also incorporates vehicles as network nodes.
The zone constraint is represented by edges between vehicles and requests: these exist only if the request originates in the same zone where the vehicle is present.
As the authors point out, the formulated min-cost flow problem can be solved in polynomial time by the network simplex algorithm.

Similarly to~\textcite{wangMinimumFleetRidesharingAware2021}, ~\textcite{quHowManyVehicles2022} also propose to incorporate ridesharing into the fleet-sizing problem. Contrary to their predecessors, however, they decouple the problem: first, they solve the ridesharing (DARP) problem, and next, they apply the method from~\textcite{vazifehAddressingMinimumFleet2018} on the graph composed from the computed ridesharing plans. 
Additionally, the matching phase computes the utility with regard to an anticipated travel demand computed from an ensemble model. 
Decoupling the problem breaks the guarantee of optimality. 
However, it means that the minimal path cover in the shareability graph can be solved in polynomial time, as in~\textcite{vazifehAddressingMinimumFleet2018}.

\subsection{Ridesharing}
\label{sec:ridesharing_intro}
Apart from fleet-sizing, shareability networks can be used to solve other problems related to the MoD.
The concept of shareability network itself has been introduced in~\textcite{santiQuantifyingBenefitsVehicle2014}.
The authors use a sharability network to study the pairwise sharability trips in Manhattan.
Note that, differently than in the fleet sizing articles, here the trips are connected in a sharability network not only in case they can be connected consecutively, but also if they can share the ride.
In conclusion, we obtain an optimal solution for the ridesharing (DARP) problem limited to two persons per vehicle, provided that we optimize the shareability network using the minimum-weighted matching, as the article proposes. 

\textcite{alonso-moraOndemandHighcapacityRidesharing2017} then lift the two-passenger limit and propose an optimal algorithm to solve the online ridesharing (DARP) problem.
However, due to the computational complexity, they introduce several heuristic relaxations to compute the solution in a reasonable time.
Later, Čáp and Alonso-Mora have analyzed the trade-off between
the operation cost and service quality with this method, which they have called the
vehicle-group assignment method (VGA)~\cite{capMultiObjectiveAnalysisRidesharing2018}.

Finally, in~\textcite{fiedlerLargescaleOnlineRidesharing2022}, the authors have examined the optimal version of the VGA algorithm from the perspectives of system efficiency and computational time. An important finding of this work is that the optimal assignments can be computed in practice only if we limit the time horizon of the instance, i.e., if we try to match only requests with origin times within a short period.
At the same time, it shows that the relaxed version performs poorly if the time horizon is longer, being outperformed by the baseline method: the insertion heuristic.
This leads us to a question: could we obtain a better solution by computing the matching over a short time horizon and then connecting the resulting plans by solving the min-weight matching over the shareability network?

\subsection{Contribution}
\label{sec:contribution}
In this paper, we extend the concept of sequential shareability with time windows, a well-known concept from fleet-sizing and DARP.
In simple words, we allow delaying plans for the purpose of enabling more connections between them if the delay does not violate the maximum delay, typically declared in the form of a time window.
The contribution is four-fold:
\begin{enumerate}
\item We \textbf{extend the formulation} of the sequential shareability with time windows.
The novel formulation is more complex, however, in line with the classical DARP, which includes time windows (Section~\ref{sec:problem_formulation}). 
\item We \textbf{design a method} that can minimize the overall cost while generating only a fraction of valid delayed plan variants (Section~\ref{sec:proposed}). 
\item We deliver a \textbf{proof of optimality} for the proposed method: we prove that given a set of plans and their maximum delay, our method generates a set of connections with a minimal cost (Section~\ref{sec:proposed}).
\item We demonstrate the capabilities of our method to solve large static DARP on \textbf{case studies} of four areas: New York City, Manhattan, Chicago, and Washington, DC (Section~\ref{sec:case_studies}).
\end{enumerate}

\section{The Plan Chaining Problem}
\label{sec:problem_formulation}
Here, we formulate the problem of \emph{chaining} plans with time windows. 
The formulation is independent of the "real-world" problem, i.e., we can use the same formulation for fleet-sizing, dispatching, and potentially other problems with a similar structure. 
The main difference between our formulation and the previous formulations~(\textcite{vazifehAddressingMinimumFleet2018,quHowManyVehicles2022}) is that we allow plan delays that respect the given time windows. 
The previous formulations chain plans with fixed plan start/end times, despite the underlying problem being some variant of vehicle routing problem with time windows.
Because the time-windows extension makes the network formulation quite complex, we do not use it as a problem formulation, in contrast to previous works~\cite{vazifehAddressingMinimumFleet2018,xuNetworkFlowBasedEfficientVehicle2022,quHowManyVehicles2022}.
Instead, in this section, we propose a more compact formulation (see Problem~\ref{prob:chaining} later in this section).
However, the network formulation, loosely connected to the shareability network introduced in~\textcite{vazifehAddressingMinimumFleet2018}, is a backbone of the proposed method, and it is presented in Section~\ref{sec:proposed}.

We say that the chaining problem is a problem of transforming given vehicles and plans to the minimum cost chains, that consist of one vehicle and one or more plans, such that all the time constraints are met, each vehicle is at the beginning of no more than one chain, and all plans are part of exactly one chain.

For the purpose of plan chaining, the internal structure of the plan (pickup and drop-off times and locations) is irrelevant. Instead, each plan $ p $ from the set of plans $ P $ is defined by the origin time $ \tor{p} $, destination time $ \tde{p} $, and maximum delay $ \dmax{p} $ which represent the plan's time window.
Each vehicle $ v $ from the set of vehicles as $ V $ is determined by its start time $ \tst{v} $. 

We assume that travel times between all plans and vehicles $ a $ and $ b $ are known in the form of a travel time function $ \ftt{a}{b}, a, b \in P \cup V $.
Note that the function parameter order matters here, as the travel time is often asymmetrical in real use cases.
Analogously, we define the travel cost function $ \fc
{a}{b} $, which represents the travel cost (if we optimize for minimum travel time, we can set $ f_{tc} = f_{tt} $).

To use the time flexibility offered by the plans' maximum delay, we introduce \emph{delayed plan variants}.
For each plan $ p $, there is a set of possible delayed plan variants $ \PV{p} $.
We denote the union of all such sets as $ \PV{} $.
Each $ p' $ in $ \PV{p} $ is defined by the delay $ \dep{p'} \leq \dmax{p} $ over the original plan.
As consequence it holds that $ \tor{p'} = \tor{p} +  \dep{p'} $ and $ \tde{p'} = \tde{p} + \dep{p'} $.

We use the term \emph{chaining} because we create sequences called \emph{plan chains}. 
\begin{defn}
\label{def:plan-chain}
A plan chain $ c =\{c_i\}_{i\in\mathbb{N}}$ is a sequence such that:
\begin{alignat}{1}
    &c_1 \in V \\ 
    &\forall c_i \in c \setminus \{c_1\},\quad c_i \in P \cup \PV{} \\ 
   &\tst{c_1} + \ftt{c_1}{c_2} \leq \tor{c_2} \\
   &\forall i \in [3, |c|],\quad  \tde{c_{i - 1}} + \ftt{c_{i - 1}}{c_i} \leq \tor{c_i} 
\end{alignat}
\end{defn}

A set of all plan chains that can be constructed from vehicles $ V $ plans $ P $ (and associated delayed plans $ \PV{} $ generated from $ P $) will be denoted as $ \mathcal{C}(P,V) $.

Finally, we define the \emph{chaining problem}: 

\begin{prob}[Plan Chaining]
\label{prob:chaining}

Given a set of plans $ P $ and a set of vehicles $ V $, compute the set of plan chains $ C $ for which the travel cost between consecutive elements is minimal:
\begin{equation}
    \min_{C \subset \mathcal{C}_P^V}\displaystyle\sum_{c \in C}\sum_{i = 2}^{|c|} \fc{c_{i - 1}}{c_i}
\end{equation}

such that exactly one variant per each plan is contained in exactly one chain and each vehicle is at the beginning of no more than one chain:
\begin{alignat}{2}
\displaystyle
&\sum_{c \in C} \sum_{i = 2}^{|c|}\mathbf{1}_{\PV{p}}(c_i) = 1, \quad &\forall p \in P \\ 
&\sum_{c \in C}[c_1 = v] \leq 1, &\forall v \in V 
\end{alignat}
\end{prob}

\subsection{Chaining Problem configuration for solving various transportation problems}
We can solve various problems related to on-demand systems by solving the chaining problem with different configurations.

For fleet-sizing, we create a dedicated vehicle for each plan and set the travel cost to a constant for every used vehicle:
\begin{equation}
    \fc{a}{b} = 
    \begin{cases}
    \text{1}, &\text{if } a \in V,  \\
    0, &\text{otherwise}.
    \end{cases}
\end{equation}
This way, the number of vehicles will be minimized without considering the travel cost between plans.

For minimum cost dispatching (DARP), we just set the $ \fc{a}{b} $ to capture the cost of travel between $ a $ and $ b $; for example, we can set:
\begin{equation}
    \fc{a}{b} = \ftt{a}{b} \quad \forall a, b \in V \cup P \cup \PV{}.
\end{equation}

To prevent long waiting times, we can set $ \fc{a}{b} = \infty $ for all $ a $ and $ b $ where the $ \tor{b} - \tde{a} \geq \delta $, where $ \delta $ is the maximum waiting time.
Analogously, we can penalize waiting times gradually by increasing the travel cost function value proportionally to the waiting time.

\section{Proposed method}
\label{sec:proposed}
In this section, we describe the proposed method and prove that it solves the chaining problem optimally.
The fleet-sizing article by~\textcite{vazifehAddressingMinimumFleet2018} formulates the fleet-sizing problem as the \emph{minimum path cover} (see Figure~\ref{fig:min_path_cover}).
\begin{figure}
    \centering
    \includegraphics[width=0.6\columnwidth]{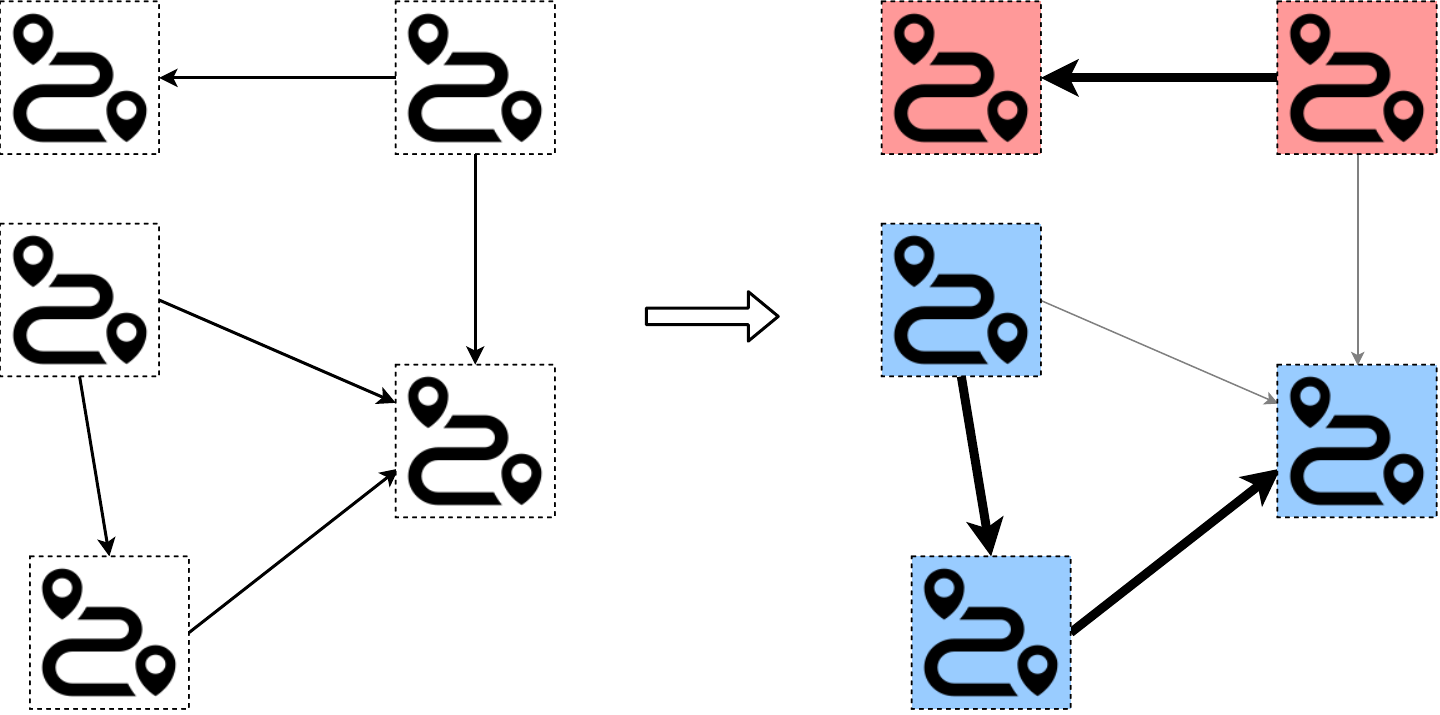}
    \caption{Example of using the minimum path cover to solve the fleet-sizing problem. 
    On the left, there is a vehicle shareability graph. 
    An arrow between any two plans signals that these plans can be served sequentially.
    On the right, there is the minimum path cover. 
    Each color represents a chain of plans to be served by one vehicle. 
    The connections (arrows) between plans in the chain are bold.
    }
    \label{fig:min_path_cover}
\end{figure}
Also, the authors conclude that the minimum path cover is equivalent to the maximum bipartite matching (see Figure~\ref{fig:maximum_matching}) and, therefore, it can be solved in polynomial time by the Hoptcroft-Karp algorithm.
\begin{figure}
    \centering
    \includegraphics[width=0.9\columnwidth]{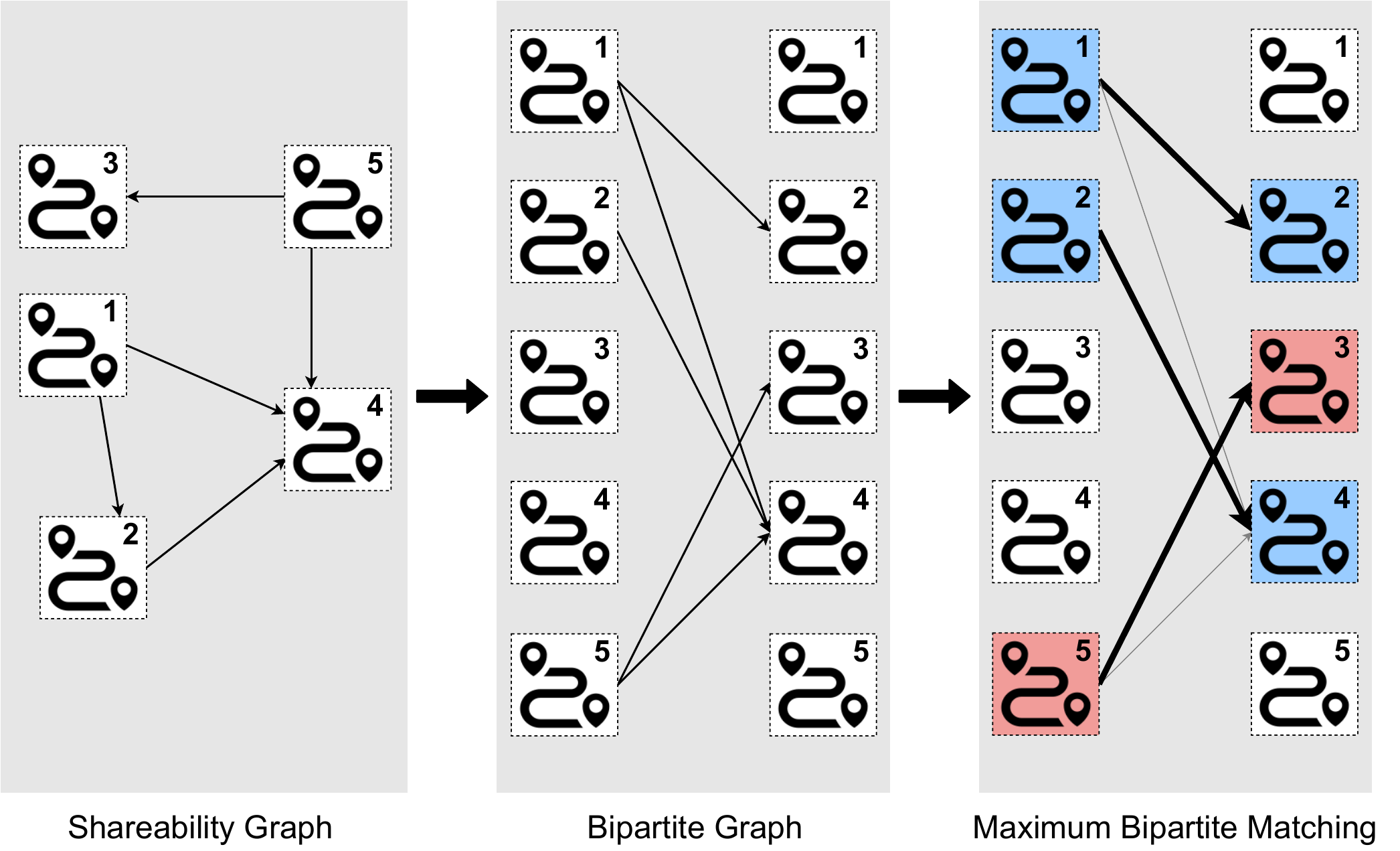}
    \caption{
    Example showing how the min path cover can be converted to the maximum bipartite matching. 
    We use the same example as in Figure~\ref{fig:min_path_cover}; plans are now numbered for clarity.
    Each color represents a chain of plans to be served by one vehicle. 
    The connections (arrows) between plans in the chain are bold.
    }
    \label{fig:maximum_matching}
\end{figure}
Analogously, we can minimize the cost of the plan chains by replacing the maximum matching with the min-cost (perfect) matching, that is, by solving the assignment problem. 
This problem can also be solved in polynomial time (e.g., by the Hungarian Algorithm).
An example of min-cost matching is shown in Figure~\ref{fig:min_cost_matching}.
\begin{figure}
    \centering
    \includegraphics[width=0.6\columnwidth]{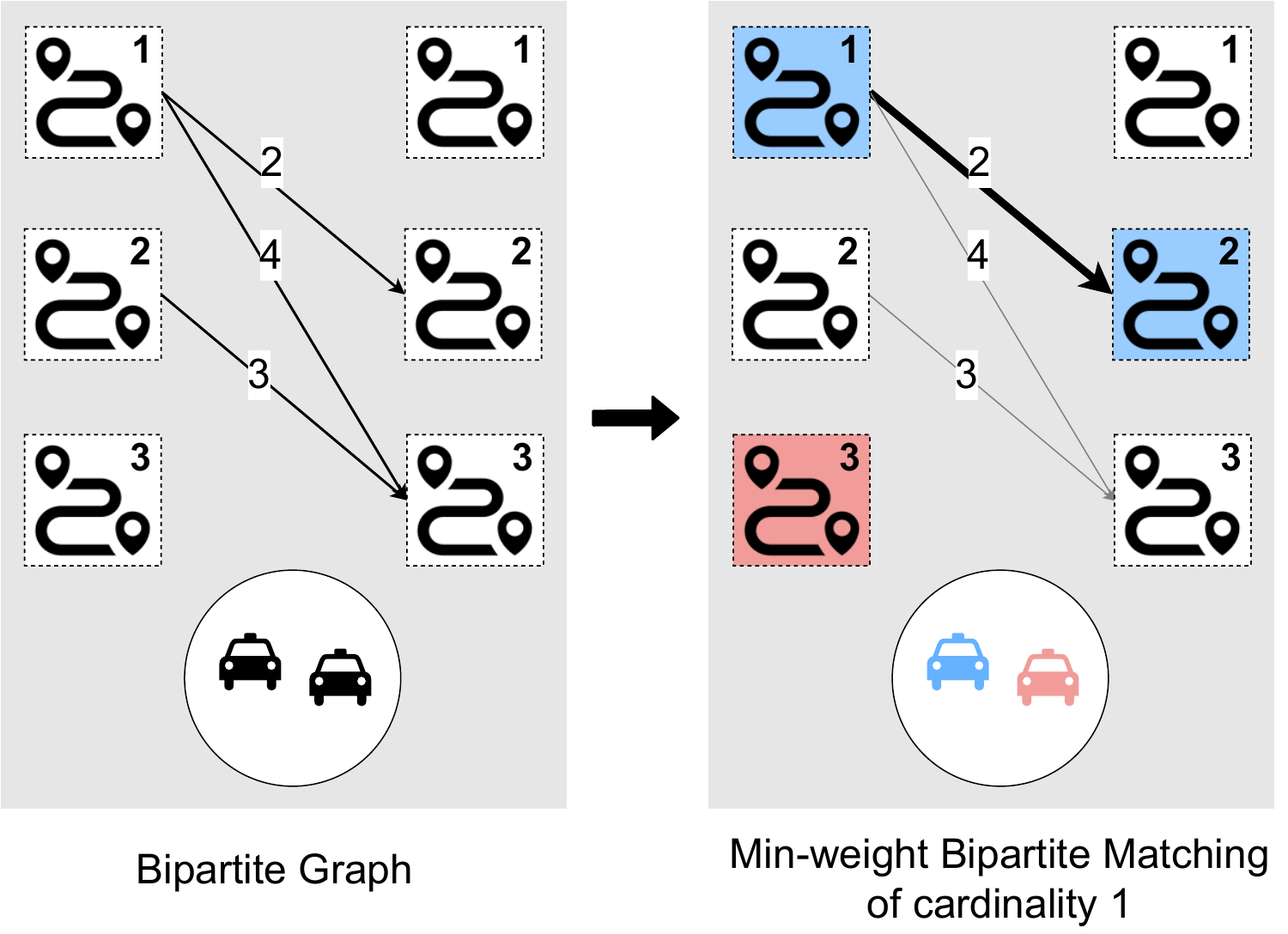}
    \caption{
    An example of plan chaining formulated as an assignment problem is a min-weight matching of a specific cardinality.
    The cardinality is given by the number of vehicles and plans: here, we have two cars available (below, in the circle) and three plans, resulting in cardinality 1 (at least one connection between plans).
    The numbers on the arcs determine the travel cost between the plans.
    Each color represents a chain of plans to be served by one vehicle. 
    The connection (arrow) between plans in the chain is bold.
    }
    \label{fig:min_cost_matching}
\end{figure}
Note that, here, we also need to represent vehicles.
For simplicity, the example assumes zero travel cost between the vehicle location and the start location of any plan.

However, even the assignment problem formulation is not suitable for optimal plan chaining, as it misses an important aspect: the time windows.
The delayed plans have a different shareability potential and, conclusively, delaying a plan results in a different bipartite graph.
To solve the chaining Problem \ref{prob:chaining}, we propose a two-step method.
First, we generate only those delayed plan variants that are necessary to guarantee the optimal solution to the chaining problem.
This variant generation process that generates only a fraction of all possible delayed plan variants is described in Section~\ref{sec:variant_generation}.
Second, we formulate the chaining problem as a constrained min-cost flow problem (MCFP).
This way, we can compute the optimal solution even for large instances despite the problem being NP-hard.
In Section~\ref{sec:mcfp_formalization}, we briefly introduce the min-cost flow problem (MCFP) formalization.
Then, in Section~\ref{sec:constrained_mcfp}, we present the constrained MCFP that represents the chaining problem and we show that an optimal solution of the proposed MCFP is an optimal solution of the corresponding chaining problem.

\subsection{Generating Plan Variants}
\label{sec:variant_generation}
In the chaining formulation, we used a set of delayed plans $ \PV{} $ for each plan $ p \in P $.
In theory, time is a continuous quantity, so the number of delayed plan variants for each plan is infinite.
However, in real-world computations, we usually consider a discrete time, which results in a finite number of plan variants. 
For example, with a resolution of one second and a plan that can be delayed by 20 seconds, there exist 20 (delayed) plan variants.

Nevertheless, as we prove further in this section, generating all possible plan variants is not necessary for guaranteeing the optimality of the plan chaining method. 
Instead, we generate only the variants that have the potential to extend the number of possible plan chains.
The algorithm we propose for generating these variants is displayed in Algorithm~\ref{alg:variant-generation}.

\begin{algorithm}
\Input{$ P $: a set of plans, $ V $: a set of vehicles}
\Output{Set of plan variants $ \PV{} $, set of connections between plans $ X $}

$ \Phi \leftarrow \{ \} $\;
$ X \leftarrow \{ \} $\;
variant\_queue $ \leftarrow $ empty queue\;

\BlankLine

\Fn{try\_connect($ a $: Plan or Vehicle, $ b $: Plan)}{
    \If{$ a \in P \cup \PV{} $}{
        min\_delay $ \leftarrow \ftt{a}{b} - (\tor{b} - \tde{a}) $ \; \label{algl:min_delay}
    }
    \Else{ 
        
        min\_delay $ \leftarrow \ftt{a}{b} - (\tor{b} - \tst{a}) $ \; 
    }
    \If{min\_delay $ \leq 0 $}{
        $ X \leftarrow X \cup (a, b) $\;
    }
    \ElseIf{min\_delay $ \leq \dmax{b} $}{
        $ \phi \leftarrow \fd{b}{min\_delay} $\;
        $ X \leftarrow X \cup (a, \phi) $\;
        $ \Phi \leftarrow \Phi \cup \phi $\;
        variant\_queue.push($ \phi $)\;
        
    }
}

\BlankLine

\For{$ a \in P \cup V $}{
    \For{$ b \in P \setminus a $}{
    	$ \ftc{a}{b}{} $\; \label{algl:ftc-plan}
	}
}
\While{variant\_queue \Not empty}{
    $ \phi \leftarrow $ variant\_queue.pop()\;
    \For{$ p \in P $}{
        \If{$ \phi $ not delayed variant of $ p $}{
            $ \ftc{\phi}{p}{} $\; \label{algl:ftc-variant}
        }
    }
}

\caption{\label{alg:variant-generation}Algorithm for generating plan variants.}
\end{algorithm}

Apart from generating variants, the algorithm also generates possible \emph{connections} between vehicles and plans.
We define the connection $ a \in P \cup V \cup \PV{} $ and $ b \in P \cup \PV{} $ as an (ordered) pair $ (a, b) $. 
Plan or vehicle $ a $ can be connected to plan $ b $ only if: 1) $ a $ is a vehicle or the non-delayed version of $ a $ is different than the non-delayed version of $ b $: $ \PV{a} \neq \PV{b} $; and 2) the time difference between them is less than or equal to the travel time between them:
\begin{equation}\label{eq:possibleConnectionsP}
\ftt{a}{b} \leq
\begin{cases}
    \tor{b} - \tde{a} & \quad \text{if } a \in P \cup \PV{} \\
    \tor{b} - \tst{a} & \quad \text{if } a \in V
\end{cases}
\end{equation}

The variant generation algorithm first tries to connect each plan and vehicle $ a $ with each plan $ b \neq a $. 
If the connection is not possible, we compute a delay for plan $ b $ as
\[
\dep{b} = \ftt{a}{b} - 
\begin{cases}
    (\tor{b} - \tde{a}) & \quad \text{if } a \in P \cup \PV{} \\
    (\tor{b} - \tst{a}) & \quad \text{if } a \in V
    
\end{cases}
\]
If $ \dep{b} \leq \dmax{b} $, we generate a delayed plan variant of $ b $ with delay $ \dep{b} $ and the corresponding connection.

Then, analogously, we try to connect each variant $ \phi \in \PV{} $ with each plan $ p \in P $ and try to delay the plan $ p $ if necessary:

\[
\dep{p} = \ftt{\phi}{p} - (\tor{p} - \tde{\phi}), 
\]

It is clear that Algorithm~\ref{alg:variant-generation} does not generate all possible plan variants.
However, we can prove that it generates all plan variants necessary to construct the optimal solution to the chaining problem.

\begin{thm}
\label{thm:variant-generation} 
There exists an optimal solution to the chaining problem that contains only non-delayed plans and plan variants generated by Algorithm~\ref{alg:variant-generation}. 
\end{thm}

Before proving Theorem~\ref{thm:variant-generation}, we will analyze some properties of Algorithm~\ref{alg:variant-generation}.
First, we assume that the connection cost does not depend on time:
\begin{ass}
\label{cor:equal-variant-cost}
All plan chains that differ only in the delay of their plan variants have equal costs.
\end{ass}

Next, note that when delaying the plan, the Algorithm~\ref{alg:variant-generation} always uses the minimum possible delay (see function $ \mathtt{try\_connect} $).

\begin{lem}
\label{lem:variant-generation-earliest-version}
The function $ \mathtt{try\_connect} $ from the variant generation algorithm always creates the connection from plan $ o $ to plan $ p $, if possible, using the minimum possible delay for plan $ p $. 
\end{lem}
\begin{proof}
We can prove this by inspecting the algorithm. 
On line~\ref{algl:min_delay}, the minimum delay is computed as a difference between travel time between $ o $ and $ p $ ($\ftt{o}{p} $) and the time difference between these plans.
We can see that this is indeed the minimum delay for a connection to be possible, as the delay is basically equal to the travel time minus the already existing time difference between plans.
On the following lines, we can see that either:
\begin{itemize}
    \item $ \mathtt{min\_delay} \leq 0 $: no delay is needed, so we keep the plan $ p $ as is. 
    As each original plan has the minimum possible delay (so we cannot assign a negative delay to a plan), this is indeed the minimum possible delay.
    \item $ 0 < \mathtt{min\_delay} \leq \dmax{p} $: in this case, we assign the minimum delay computed as described above.
    \item $ \mathtt{min\_delay} > \dmax{p} $: connection is impossible as the time constraints of plan $ p $ would be violated.
\end{itemize}
Thus, the claim is proved.
\end{proof}

Finally, note that for a sequence of plans to be a plan chain, all its sub-sequences have to be also plan chains, as stated in the following.
\begin{lem}
\label{lem:sub-chainings}
If we remove a plan from the end of a plan chain, the resulting sequence is also a plan chain.
\end{lem}
\begin{proof}
It trivially comes from Definition~\ref{def:plan-chain}.
\end{proof}

Now, we can go back to Theorem~\ref{thm:variant-generation}:
\begingroup
\def\thethm{\ref{thm:variant-generation}}
\begin{thm}
An optimal solution to the chaining problem exists that contains only non-delayed plans and plan variants generated by Algorithm~\ref{alg:variant-generation}.
\end{thm}
\addtocounter{thm}{-1}
\endgroup

\begin{proof}[Proof of Theorem \ref{thm:variant-generation}]
\label{proof:variant_generation_generates_all_variants}
We will proceed by induction and show that for each possible plan chain, Algorithm~\ref{alg:variant-generation} generates a set of connections that compose an equivalent plan chain that differs only in the plan delays (and in conclusion, has an equal cost, see Assumption~\ref{cor:equal-variant-cost}).

Let us start with a chain of length one: a vehicle.
As the algorithm does not discard any vehicles, it is clear that all chains of length one are covered.
Let a chain of length two be given. Thus, we have a vehicle $ v $ and a plan $ p $, and we try to connect them by a connection $ (v, p) $.
There are two possible cases:
\begin{enumerate}
    \item The connection $ (v, p) $ is not possible: we cannot generate the plan chain.
    \item The connection $ (v, p) $ is possible: as we can see, all such connections are created by Algorithm~\ref{alg:variant-generation} (lines line~\ref{algl:ftc-plan}~and~\ref{algl:ftc-variant}). 
\end{enumerate}
Analogously, the chains of length 3 (a vehicle and two plans) are generated by Algorithm~\ref{alg:variant-generation}.
Let us now assume that Algorithm~\ref{alg:variant-generation} generates all possible plan chains of length $ k\leq n $, and suppose we want to connect chain $(p_1, \cdots, p_{n+1})$. 
On the account of Lemma~\ref{lem:sub-chainings}, the problem boils down to connecting the chain $ (v, p_1, \cdots, p_n) $ to (possibly delayed) plan $ p_{n + 1} $.
There are three possible cases:
\begin{enumerate}
    \item The connection $ (p_n, p_{n + 1}) $ is not possible: we cannot generate the plan chain.
    \item The connection $ (p_n, p_{n + 1}) $ is possible: as we can see, all such connections are created by line~\ref{algl:ftc-plan} of Algorithm~\ref{alg:variant-generation} in case that $ \dep{p_{n + 1}} = 0 $; by line~\ref{algl:ftc-variant} otherwise.
    \item The connection $ (p_n, p_{n + 1}) $ would be possible only with an earlier variant of $ p_n $.
    However, according to Lemma~\ref{lem:variant-generation-earliest-version} and by induction hypothesis, we know that $ p_n $ is computed by the algorithm and is the earliest variant that can be connected to  $ p_{n - 1} $. Therefore the chain $ (v, p_1, \cdots, p_n, p_{n + 1}) $ is not possible.
\end{enumerate}

Thus, the proof is complete.
\end{proof}

\subsection{Min Cost Flow Problem Formalization}
\label{sec:mcfp_formalization}
The minimum-cost flow problem (MCFP, see~\textcite{ahujaNetworkFlowsTheory1993}) is an optimization problem of finding a minimum set of flows through a flow network that is in total equal to a specified amount of flow.

A \emph{flow network} is a directed graph $ G = (\vertices, E) $.
Each edge $ e $ from $ E $ is defined as a 4-tuple $ (f_e, l_e, u_e, c_e) $, where:
\begin{itemize}
    \item $ f_e $ is the flow of the edge $ e $, a variable,
    \item $ l_e $ is the lower bound of the edge $ e $, a constraint,
    \item $ u_e $ is the upper bound of the edge $ e $, a constraint,
    \item $ c_e $ is the cost of the edge $ e $, a constant.
\end{itemize}

Each vertex $ \vertex $ from a set $ \vertices $ has an associated supply value $ s_{\vertex} $. 
If $ s_{\vertex} > 0 $, the node is called a \emph{source}, if $ s_{\vertex} = 0 $ it is a \emph{transshipment node}, and if $ s_{\vertex} < 0 $ it is called a \emph{sink}.
We mark the set of edges originating in the vertex $ \kappa $ as $ O_{\vertex} $ and a set of edges with a destination in $ \vertex $ as $ D_{\vertex} $.

The min-cost network flow problem is then formulated as:

Minimize:

\begin{equation}
    \sum_{e \in E} c_e f_e,
\end{equation}

subject to 

\begin{equation}
\label{eq:mcfp-flow-conservation}
    \sum_{e \in O_{\vertex}} f_e - \sum_{e \in D_{\vertex}} f_e = s_{\vertex} \quad \forall \vertex \in \vertices,
\end{equation}

where $ f_e \in [l_e, u_e] \quad \forall e \in E $.
Here, Equation~\ref{eq:mcfp-flow-conservation} is the flow conservation constraint.

\subsection{Chaining as the Min-cost Flow Problem}
\label{sec:constrained_mcfp}
To find an optimal solution to the chaining problem given the vehicles, plans, and delayed plan variants and connections generated by Algorithm~\ref{alg:variant-generation}, we propose a constrained min-cost flow problem formulation exemplified in Figure~\ref{fig:chaining_example}.

\begin{figure}
    \centering
    \includegraphics[width=1\linewidth]{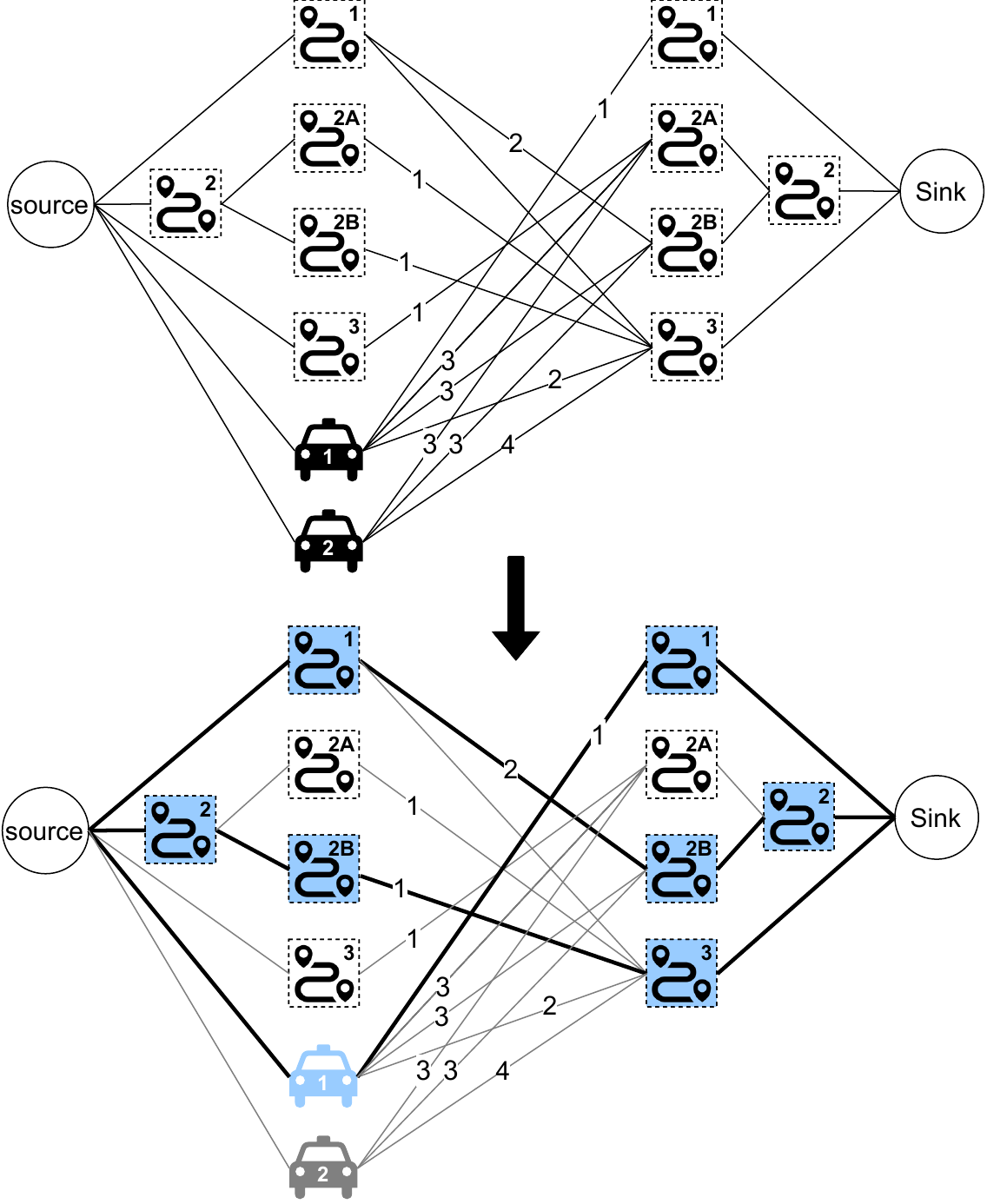}
    \caption{
    Example of the final chaining formulation formulated as an MCFP. 
    The vertices are, from the left: the source, then left plan vertices for each plan, left variant vertices for each variant (e.g., 2B translates to plan 2, variant B), and vehicle vertices, then right plan and variant vertices, and finally, the sink.
    If a plan has no delayed variants, there is no variant vertex.
    The travel cost is expressed by the number over the arc.
    Arcs without numbers have zero cost.
    The inflow of the source is equal to the number of plans, and the outflow of the sink is reversed to that.
    In the bottom image, there is the solution marked by bold arcs (used arcs with active flow). 
    For readability, vertices between solution arcs are painted blue.
    }
    \label{fig:chaining_example}
\end{figure}

There are seven types of nodes:
\begin{itemize}
    \item A single source and a single sink node,
    \item left and right plan nodes for each original plan,
    \item left and right variant nodes generated for plans with delayed variants.
    \item vehicle nodes
\end{itemize}
Each node has a zero supply, with the exception of the source node, which has a supply equal to the number of plans (three in the example), and the sink node, which has the inverse supply.
The edges are generated:
\begin{itemize}
    \item from source to each left plan node and each vehicle node,
    \item from each left plan node to each of its corresponding variant nodes,
    \item from each right variant node to its corresponding right plan node
    \item from each right plan node to sink,
    \item and from vehicles and left plans/variants to right plans/variants according to the output of the Algorithm~\ref{alg:variant-generation}.
\end{itemize}
Each edge has a direction from left to right. 
Also, for each edge $ e $, $ l_e = 0 $, and $ u_e = 1 $.
The cost of each edge is zero, with the exception of the edges between left plan/variant nodes and right plan/variant nodes that have a cost equal to the travel time between the destination of the previous plan and the origin of the next plan. 


To solve the chaining problem optimally, we solve this min-cost flow problem and then connect (chain) plans for each edge between the left plan/variant node and the right plan/variant node with active flow (flow $ = 1 $).
For example, in the solution illustrated in Figure~\ref{fig:chaining_example} by bold lines, the edges (vehicle 2, plan 1), (plan 1, plan 2B), and (plan 2B, plan 3) are part of the solution. 
Therefore, the resulting solution is a single plan chain (vehicle 2, plan 1, plan 2B, plan 3).

To guarantee that this process results in a feasible system plan, we need to constrain the min-cost flow problem such that a) only one variant has an active connection on the left part of the flow problem, b) only one variant has an active connection on the right part, and c) the connected variant on the left is the same as the one connected on the right.
Conveniently, conditions a) and b) are guaranteed by the flow conservation constraint. 
Next, we have to introduce a constraint to guarantee the condition c). We need two symmetric constraints:
\begin{equation}
    f(\vertex_l^p, \vertex_l^{\phi}) + \sum_{\substack{\phi' \in \PV{p} \\ \phi' \not\equiv \phi}} f(\vertex_r^{\phi'}, \vertex_r^p) \leq 1 \quad \forall p \in P, \forall \phi \in \PV{p}
\end{equation}
\begin{equation}
    f(\vertex_r^{\phi}, \vertex_r^p) + \sum_{\substack{\phi' \in \PV{p} \\ \phi' \not\equiv \phi}} f(\vertex_l^p, \vertex_l^{\phi'}) \leq 1 \quad \forall p \in P, \forall \phi \in \PV{p} 
\end{equation}
Here $ \vertex_l^p $ is the left plan vertex connected to the left variant vertex $ \vertex_l^{\phi} $, and $ \vertex_r^p $ is the right plan vertex connected to the right variant vertex $ \vertex_r^{\phi} $, 
As these constraints are generated for each variant node, there are no such constraints for plans without delayed plan variants.
Note that with these constraints, we lose the ability to use the LP solver because the flows could now be non-integer. 
Therefore, we have to solve this formulation using an ILP solver.

From the description above, it results that the proposed MCFP corresponds to problem~\ref{prob:chaining}.
Therefore, by optimally solving the MCFP, we obtain an optimal solution to the chaining problem.

\section{Demonstration: Chaining DARP Plans}
\label{sec:case_studies}
To demonstrate the possible application of the proposed chaining method, we use it as a component of a heuristic method for solving the ridesharing dial-a-ride problem (DARP).
The DARP is a problem of finding a set of vehicle plans that serve a given transportation demand with a given vehicle fleet while minimizing the total cost of the plans (usually the total travel time) and respecting the problem constraints (time windows, vehicle capacity, etc.).
There are three main groups of methods for solving DARP. 
First, there are exact methods, which can find the optimal solution to the problem but are computationally expensive and thus suitable only for specific instances.
An example of such a method is the Vehicle-group assignment (VGA) method~\cite{alonso-moraOndemandHighcapacityRidesharing2017}, or the branch-and-price scheme typically used in operational research literature~\cite{ropkeBranchCutPrice2009}.
Second, there are constructive heuristics, which create a suboptimal solution but are computationally efficient and typically can be used for any DARP instance.
Insertion heuristics~\cite{jawHeuristicAlgorithmMultivehicle1986} is an example of such a method.
Finally, there are metaheuristics, which are iterative methods that start with a suboptimal solution (usually obtained by solving the problem with some construction heuristic) and improve it over time.
Variable Neighborhood Search (VNS)~\cite{muelasVariableNeighborhoodSearch2013}, Adaptive Large Neighborhood Search (ALNS)~\cite{pfeifferALNSAlgorithmStatic2022}, and Genetic Algorithms~\cite{genhongImprovedGroupingGenetic2014} or their combinations~\cite{belhaizaHybridAdaptiveLarge2019} counts among popular methods in this category.

The heuristic we propose here for solving DARP belongs to the second category.
It divides the instance into batches that are solved optimally, and then the resulting plans are chained using the proposed method.
By dividing the instance into batches and chaining the plans later, we can avoid the computational complexity resulting from the long time horizon of the instance, reported previously in literature~\cite{fiedlerLargescaleOnlineRidesharing2022}.

We start by introducing the problem and the scheme of the proposed method based on optimal chaining (Section~\ref{sec:vga_chaining}).
Then, in Section~\ref{sec:instances}, we described briefly the instance and methods used in this demonstration, and finally, Section~\ref{sec:results} presents the results.

\subsection{A New Heuristic Method for Solving DARP Based on Optimal Chaining}
\label{sec:vga_chaining}
As we mentioned in the Introduction, one of the optimal methods for solving the ridesharing DARP, the Vehicle-group assignment (VGA) method \cite{alonso-moraOndemandHighcapacityRidesharing2017}, suffers from high computational complexity when solving instances spanning time periods longer than tens of seconds~\cite{fiedlerLargescaleOnlineRidesharing2022}.
Therefore, for offline ridesharing, where the task is to solve instances spanning a longer time, a different method is required.
In this demonstration, we propose a new heuristic method using the VGA method as its component. 
Unlike classical metaheuristics (see review in~\textcite{hoSurveyDialarideProblems2018}), which are iterative methods, the heuristic proposed here is constructive: it creates a single solution and does not improve it afterward. 

The principle of the proposed heuristic method is to split the demand by time into short time intervals (\emph{batches}) that can be solved by the VGA method optimally. 
Then, these batches are joined by the chaining method proposed in this work, resulting in plans that cover the whole instance period.
The simplified scheme of the method is in Figure~\ref{fig:main_diagram}.
\begin{figure}
    \centering
    \includegraphics[width=1\columnwidth]{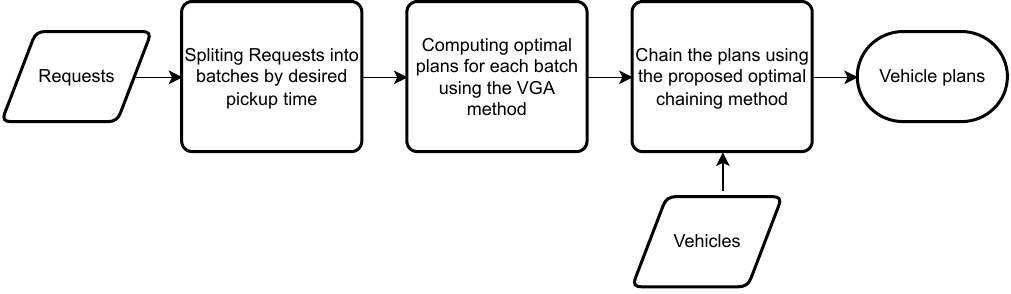}
    \caption{The scheme of the proposed method.}
    \label{fig:main_diagram}
\end{figure}
Note that in this scheme, the VGA method can be replaced by any other optimal DARP solution method.
If you are interested in more details about how the VGA method needs to be modified in the proposed heuristic, check the supplementary material.

\subsection{Instances, Evaluated Methods, and Experiment Configuration}
\label{sec:instances}
For the demonstration, we used large-scale ridesharing DARP instances presented in~\textcite{fiedlerLargescaleRidesharingDARP2023}
These instances are created from the publicly available datasets for three US cities: New York City, Chicago, and Washington, DC.
In addition to these three, these datasets contain the Manhattan area created using the New York City demand.
These four areas vary not only geographically but are also very different in the magnitude of both road network and travel demand (see Table~\ref{tab:areas}).
\begin{table}
    \centering
    \caption{Area statistics, as specified in~\textcite{fiedlerLargescaleRidesharingDARP2023}.}
    \label{tab:areas}
    \settowidth\tymin{Manhattan}
    \setlength\extrarowheight{2pt}
    \begin{tabulary}{\linewidth}{LRRRR}
    Area & Road length [km] & Area [km\textsuperscript{2}] & Requests/h/km\textsuperscript{2} \\
    \hline
    NYC &  27721 & 1508 & 26.72 \\
    Manhattan &  1329 & 87 & 266.87 \\
    Chicago  & 31982 & 1004 & 1.13 \\
    DC & 5877 & 181 & 3.65 \\
    \end{tabulary}
\end{table}
For each area, we selected only instances with a maximum delay of \SI{5}{minutes} and instance duration of 5, 15, and \SI{30}{minutes}.
We consider the request to be served by conventional personal vehicles and, therefore, we set the vehicle capacity to four persons.
The parameters of all used instances are displayed in Table~\ref{tab:instances}
\begin{table}
\caption{Instances}
\centering 
\begin{tabular}{lrrrr}
Instances & Duration & Requests & Vehicles & Trip Length [mean$\pm$std min] \\
DC & 5 min & 54 & 50 & 15.2$\pm$9.8min \\
DC & 15 min & 163 & 121 & 15.8$\pm$8.7min \\
DC & 30 min & 328 & 180 & 15.9$\pm$8.4min \\
Chicago & 5 min & 91 & 65 & 14.2$\pm$16.6min \\
Chicago & 15 min & 274 & 198 & 13.4$\pm$15.1min \\
Chicago & 30 min & 596 & 299 & 14.1$\pm$15.9min \\
Manhattan & 5 min & 1658 & 781 & 7.8$\pm$4.0min \\
Manhattan & 15 min & 5113 & 1672 & 7.7$\pm$3.8min \\
Manhattan & 30 min & 10362 & 1993 & 7.6$\pm$3.8min \\
NYC & 5 min & 3488 & 2384 & 10.1$\pm$6.8min \\
NYC & 15 min & 10567 & 5085 & 9.9$\pm$6.6min \\
NYC & 30 min & 20841 & 5905 & 9.9$\pm$6.5min \\
\end{tabular}
\label{tab:instances}
\end{table}
    

We compared the proposed heuristic method with three other methods for solving DARP, each representing a single category of DARP solution methods:
\begin{itemize}
    \item the vehicle Vehicle-group Assignment (VGA) method~\cite{alonso-moraOndemandHighcapacityRidesharing2017,capMultiObjectiveAnalysisRidesharing2018}: an optimal solution method,
    \item the Insertion Heuristic (IH)~\cite{jawHeuristicAlgorithmMultivehicle1986}, implemented as in~\textcite{fiedlerLargescaleOnlineRidesharing2022}, representing simple construction heuristics, 
    \item and the Adaptive Large Neighborhood Search (ALNS)~\cite{ropkeAdaptiveLargeNeighborhood2006}, configured as in~\textcite{masmoudiHybridAdaptiveLarge2020}\footnote{We omitted the "hybrid" part as implementing the genetic operators for free-floating vehicles would be too complicated, considering that we use this metaheuristic only as one of three baseline methods.}.
\end{itemize}

As the selected set of instances is challenging for the solution methods due to the large scale, we cannot expect that all the solution methods will be able to compute all instances. 
However, we include even the optimal method with exponential complexity to cover all major types of DARP solution methods.
Considering the proposed method, we tested multiple of its configurations. 
First, we tried batch lengths of 30, 6, 120, 240, and \SI{480}{s}.
Second, we also tried a time-limited version where the underlying VGA method does not solve the individual batches optimally but is time-limited instead.

We implemented all methods in C++. 
For solving mixed integer programs, we used the Gurobi~\cite{gurobi}.
We ran the experiments on the AMD EPYC 7543 CPU.
The computation was limited to \SI{24}{hours}.
Because some of the methods also require an extensive amount of memory, we limited the RAM usage to \SI{80}{GB}, not including the memory used to store the distance matrix for travel time computation (which varies dramatically between areas).
As the proposed method can be easily parallelized, we used multiple threads to solve the method.
Note, however, that the parallelization is limited to the trivially parallelizable algorithms (computing batches in parallel) or to the library calls with built-in parallelization (Gurobi solver).
We believe that the evaluated methods can be further parallelized to achieve a smaller computational time.
Finally, we did not spend extensive time with software optimizations of the evaluated methods.
Therefore, the computational time results should be taken with a grain of salt.

\subsection{Results}
\label{sec:results}
We present here the results in four tables: one for each area.
The areas are discussed from the most complex (largest by demand).
In each area table, only the methods that were able to solve at least one instance for the area are included.

Table~\ref{tab:results_nyc} shows the results for New York City. 
Only the variants of the chaining method, together with IH, were able to solve any instance within the time limit. 
Moreover, some of the variants of the proposed method were not able to solve the instance either, and the \SI{30}{minute} instance was solved only by the IH.
However, for the remaining two instances, the proposed method provides the best results.
\begin{table}
\caption{NYC results. 
None of the NYC instances were solved by the optimal method in time.
The same is true for the metaheuristic variants}
\footnotesize
\centering 
\makebox[\textwidth]{
\setlength{\tabcolsep}{2pt}
\begin{tabular}{lllrrrrrrrrr}
\toprule
 &  &  & \multicolumn{3}{c}{Total Cost [min]} & \multicolumn{3}{c}{Used Vehicles} & \multicolumn{3}{c}{Comp. time [s]} \\
Method & Batch [s] & Lim. & \SI{5}{min} & \SI{15}{min} & \SI{30}{min} & \SI{5}{min} & \SI{15}{min} & \SI{30}{min} & \SI{5}{min} & \SI{15}{min} & \SI{30}{min}  \\
\midrule
IH & - & no & 30397 & 81486 & \bfseries 157023 & 1613 & \bfseries 3716 & \bfseries 5031 & \bfseries 5.76 & \bfseries 39.55 & \bfseries 91.87 \\
\cline{1-12} \cline{2-12}
\multirow[c]{6}{*}{\rotatebox{90}{Proposed met.}} & 30 & no & 34607 & 100729 & - & 2061 & 4812 & - & 61.68 & 1004.12 & - \\
\cline{2-12}
 & 60 & no & 31789 & 93646 & - & 1854 & 4572 & - & 43.26 & 709.80 & - \\
\cline{2-12}
 & \multirow[c]{2}{*}{120} & no & 29364 & 85486 & - & 1680 & 4269 & - & 249.57 & 897.20 & - \\
 &  & yes & 29372 & 85487 & - & 1680 & 4266 & - & 135.45 & 719.93 & - \\
\cline{2-12}
 & 240 & yes & 27509 & - & - & 1567 & - & - & 224.84 & - & - \\
\cline{2-12}
 & 480 & yes & \bfseries 25736 & \bfseries 76832 & - & \bfseries 1450 & 4308 & - & 223.54 & 1020.74 & - \\
\cline{1-12} \cline{2-12}
\bottomrule
\end{tabular}
}
\label{tab:results_nyc}
\end{table}

Similarly, in the Manhattan instance, (Table~\ref{tab:results_man}), the proposed method beats the IH in 2 of 3 instances.
Here, the proposed method was able to compute even the instance with the longest time horizon; however, the successful variants provide worse results than IH.
The ALNS variants were able to compute at least the \SI{5}{minute} instance here, but they were outperformed by the proposed method. 
\begin{table}
\caption{Manhattan results. 
Here, almost all variants of all methods provide results for at least one instance. 
The ALNS-IH method is ALNS initialized with the solution of the insertion heuristic.
The ALNS-prop is the ALNS initialized with the proposed method of batch \SI{120}{s}.
}
\footnotesize
\centering 
\makebox[\textwidth]{
\setlength{\tabcolsep}{4pt}
\begin{tabular}{lllrrrrrrrrr}
\toprule
 &  &  & \multicolumn{3}{c}{Total Cost [min]} & \multicolumn{3}{c}{Used Vehicles} & \multicolumn{3}{c}{Comp. time [s]} \\
Method & Batch [s] & Lim. & \SI{5}{min} & \SI{15}{min} & \SI{30}{min} & \SI{5}{min} & \SI{15}{min} & \SI{30}{min} & \SI{5}{min} & \SI{15}{min} & \SI{30}{min}  \\
\midrule
ALNS & - & no & 8533 & - & - & 516 & - & - & 28891.51 & - & - \\
\cline{1-12} \cline{2-12}
ALNS-IH & - & no & 8371 & - & - & 512 & - & - & 25707.88 & - & - \\
\cline{1-12} \cline{2-12}
ALNS-prop & - & no & 7744 & - & - & \bfseries 505 & - & - & 27655.04 & - & - \\
\cline{1-12} \cline{2-12}
IH & - & no & 9626 & 25386 & \bfseries 48420 & 594 & \bfseries 1306 & \bfseries 1721 & \bfseries 0.61 & \bfseries 6.09 & \bfseries 18.40 \\
\cline{1-12} \cline{2-12}
\multirow[c]{6}{*}{\rotatebox{90}{Proposed met.}} & 30 & no & 10426 & - & - & 738 & - & - & 3.96 & - & - \\
\cline{2-12}
 & 60 & no & 9140 & 27686 & 55533 & 635 & 1648 & 1993 & 3.45 & 47.90 & 1706.07 \\
\cline{2-12}
 & \multirow[c]{2}{*}{120} & no & 8312 & 24360 & 48662 & 568 & 1504 & 1963 & 144.69 & 273.49 & 845.55 \\
 &  & yes & 8305 & 24360 & 48662 & 568 & 1504 & 1963 & 107.69 & 274.43 & 835.12 \\
\cline{2-12}
 & 240 & yes & 7764 & \bfseries 21646 & - & 547 & 1383 & - & 205.15 & 625.21 & - \\
\cline{2-12}
 & 480 & yes & \bfseries 7445 & - & - & 562 & - & - & 205.90 & - & - \\
\cline{1-12} \cline{2-12}
\bottomrule
\end{tabular}
}
\label{tab:results_man}
\end{table}
Note, however, that the ANLS-VGA variant is outperformed only by the proposed method with the batch length of \SI{480}{s} which actually does not chain anything, as the batch is longer than the instance time horizon. 
The reason why this method is able to compute the instance, while the VGA is failing is probably because the proposed method computes the VGA part without considering the positions of the vehicles, even for the first batch. 
This leads to suboptimality, but also to the reduction of CPU time and memory requirements.

In the Chicago area, the shortest instance was solved to optimality, showing the limits for practical application of the heuristic methods.
All other methods were able to solve all instances, with the exception of the 30 and \SI{60}{s} variants of the proposed method.
The best solution for the \SI{15}{minutes} and \SI{30}{minute} instances, is provided by ALNS.
For the \SI{15}{minute} instance, however, it is the ALNS variant initialized by the proposed method.
\begin{table}
\caption{Chicago results.
Here, almost all methods compute all the instances.
The ALNS-IH method is ALNS initialized with the solution of the insertion heuristic.
The ALNS-prop is the ALNS initialized with the proposed method of batch \SI{120}{s}.}
\footnotesize
\centering 
\makebox[\textwidth]{
\setlength{\tabcolsep}{4pt}
\begin{tabular}{lllrrrrrrrrr}
\toprule
 &  &  & \multicolumn{3}{c}{Total Cost [min]} & \multicolumn{3}{c}{Used Vehicles} & \multicolumn{3}{c}{Comp. time [s]} \\
Method & Batch [s] & Lim. & \SI{5}{min} & \SI{15}{min} & \SI{30}{min} & \SI{5}{min} & \SI{15}{min} & \SI{30}{min} & \SI{5}{min} & \SI{15}{min} & \SI{30}{min}  \\
\midrule
ALNS-IH & - & no & 1165 & 2998 & \bfseries 6612 & 42 & \bfseries 97 & \bfseries 166 & 57.58 & 699.08 & 10049.92 \\
\cline{1-12} \cline{2-12}
ALNS-prop & - & no & 1153 & \bfseries 2982 & 6690 & \bfseries 41 & 102 & 170 & 56.85 & 723.35 & 10467.21 \\
\cline{1-12} \cline{2-12}
IH & - & no & 1200 & 3117 & 7090 & 46 & 114 & 209 & \bfseries 0.01 & \bfseries 0.04 & \bfseries 0.15 \\
\cline{1-12} \cline{2-12}
VGA & - & no & \bfseries 1129 & - & - & 42 & - & - & 15.14 & - & - \\
\cline{1-12} \cline{2-12}
\multirow[c]{8}{*}{\rotatebox{90}{Proposed method}} & 30 & no & 1451 & 3940 & - & 57 & 165 & - & 0.12 & 0.62 & - \\
\cline{2-12}
 & 60 & no & 1411 & 3825 & - & 55 & 153 & - & 0.11 & 0.40 & - \\
\cline{2-12}
 & \multirow[c]{2}{*}{120} & no & 1311 & 3613 & 8943 & 51 & 142 & 258 & 0.14 & 0.32 & 1.84 \\
 &  & yes & 1311 & 3613 & 8943 & 51 & 142 & 258 & 0.11 & 0.35 & 1.84 \\
\cline{2-12}
 & \multirow[c]{2}{*}{240} & no & 1234 & 3334 & 8246 & 48 & 129 & 240 & 0.12 & 0.42 & 1.38 \\
 &  & yes & 1234 & 3334 & 8246 & 48 & 129 & 240 & 0.13 & 0.46 & 1.35 \\
\cline{2-12}
 & \multirow[c]{2}{*}{480} & no & 1132 & 3107 & 7391 & \bfseries 41 & 112 & 212 & 0.36 & 10.05 & 32.18 \\
 &  & yes & 1132 & 3107 & 7391 & \bfseries 41 & 112 & 212 & 0.37 & 10.15 & 32.21 \\
\cline{1-12} \cline{2-12}
\bottomrule
\end{tabular}
}
\label{tab:results_chic}
\end{table}

In the DC instance, we were able to solve all instances optimally.
Therefore, comparing the heuristic methods is irrelevant.
However, we can observe that for the \SI{5}{minute} instance, the cost of the heuristic solutions of some methods is very close to the optimal solution.
This signalizes that when the solution space is smaller, the heuristic methods perform well compared to the more complex instances.
\begin{table}
\caption{DC results.
The ALNS-IH method is ALNS initialized with the solution of the insertion heuristic.
The ALNS-prop is the ALNS initialized with the proposed method of batch \SI{120}{s}.
}
\footnotesize
\centering 
\makebox[\textwidth]{
\setlength{\tabcolsep}{4pt}
\begin{tabular}{lllrrrrrrrrr}
\toprule
 &  &  & \multicolumn{3}{c}{Total Cost [min]} & \multicolumn{3}{c}{Used Vehicles} & \multicolumn{3}{c}{Comp. time [s]} \\
Method & Batch [s] & Lim. & \SI{5}{min} & \SI{15}{min} & \SI{30}{min} & \SI{5}{min} & \SI{15}{min} & \SI{30}{min} & \SI{5}{min} & \SI{15}{min} & \SI{30}{min}  \\
\midrule
ALNS & - & no & 1062 & 2843 & 5252 & 39 & 85 & 127 & 19.21 & 149.25 & 979.12 \\
\cline{1-12} \cline{2-12}
ALNS-IH & - & no & 1023 & 2789 & 5157 & \bfseries 38 & 88 & 132 & 18.31 & 137.85 & 843.49 \\
\cline{1-12} \cline{2-12}
ALNS-prop & - & no & 1029 & 2789 & 5159 & \bfseries 38 & 87 & 129 & 18.25 & 150.52 & 1002.02 \\
\cline{1-12} \cline{2-12}
IH & - & no & 1049 & 2934 & 5593 & 39 & 93 & 146 & \bfseries 0.00 & \bfseries 0.01 & \bfseries 0.04 \\
\cline{1-12} \cline{2-12}
VGA & - & no & \bfseries 1009 & \bfseries 2677 & \bfseries 4821 & \bfseries 38 & \bfseries 82 & \bfseries 122 & 0.09 & 2.31 & 457.56 \\
\cline{1-12} \cline{2-12}
\multirow[c]{8}{*}{\rotatebox{90}{Proposed method}} & 30 & no & 1136 & 3545 & 6859 & 46 & 121 & 180 & 0.11 & 0.27 & 1.51 \\
\cline{2-12}
 & 60 & no & 1122 & 3496 & 6744 & 45 & 120 & 180 & 0.10 & 0.36 & 1.29 \\
\cline{2-12}
 & \multirow[c]{2}{*}{120} & no & 1095 & 3373 & 6400 & 44 & 115 & 170 & 0.08 & 0.22 & 1.05 \\
 &  & yes & 1095 & 3373 & 6400 & 44 & 115 & 170 & 0.06 & 0.20 & 1.05 \\
\cline{2-12}
 & \multirow[c]{2}{*}{240} & no & 1023 & 3263 & 6095 & 39 & 113 & 169 & 0.06 & 0.22 & 0.61 \\
 &  & yes & 1023 & 3263 & 6095 & 39 & 113 & 169 & 0.09 & 0.23 & 0.72 \\
\cline{2-12}
 & \multirow[c]{2}{*}{480} & no & 1015 & 2983 & 5540 & \bfseries 38 & 104 & 154 & 0.09 & 0.18 & 0.44 \\
 &  & yes & 1015 & 2983 & 5540 & \bfseries 38 & 104 & 154 & 0.08 & 0.18 & 0.43 \\
\cline{1-12} \cline{2-12}
\bottomrule
\end{tabular}
}
\label{tab:results_dc}
\end{table}

If we look at the computational time, it is clear absolute winner is the insertion heuristic.
We can see that with an increased instance time horizon, the computational time grows more than linearly.
Compared to it, the optimal method is slower, and the computational time grows even faster, exponentially with the instance time horizon.
When we compare VGA and the proposed method, the proposed method is mostly faster, and the computational time does not grow that fast.
However, we can see differences between method variants: longer batches result in a linear computation time growth, while shorter batches slow down faster. 
This is a result of the chaining, which is not parallelized, and is much more difficult when the batches are short.
Finally, the metaheuristic is the slowest one, by far. 
It could be probably sped up by reducing the number of iterations but with the cost of reducing the quality of the solutions.

The last measured quality in the main result table is the number of used vehicles. 
This represents the capital cost of the solution, and also somehow indicates the vehicle occupancy, which is plotted separately in Figure~\ref{fig:occupancy} in the form of a histogram.
\begin{figure}
    \centering
    \includegraphics[width=1\textwidth]{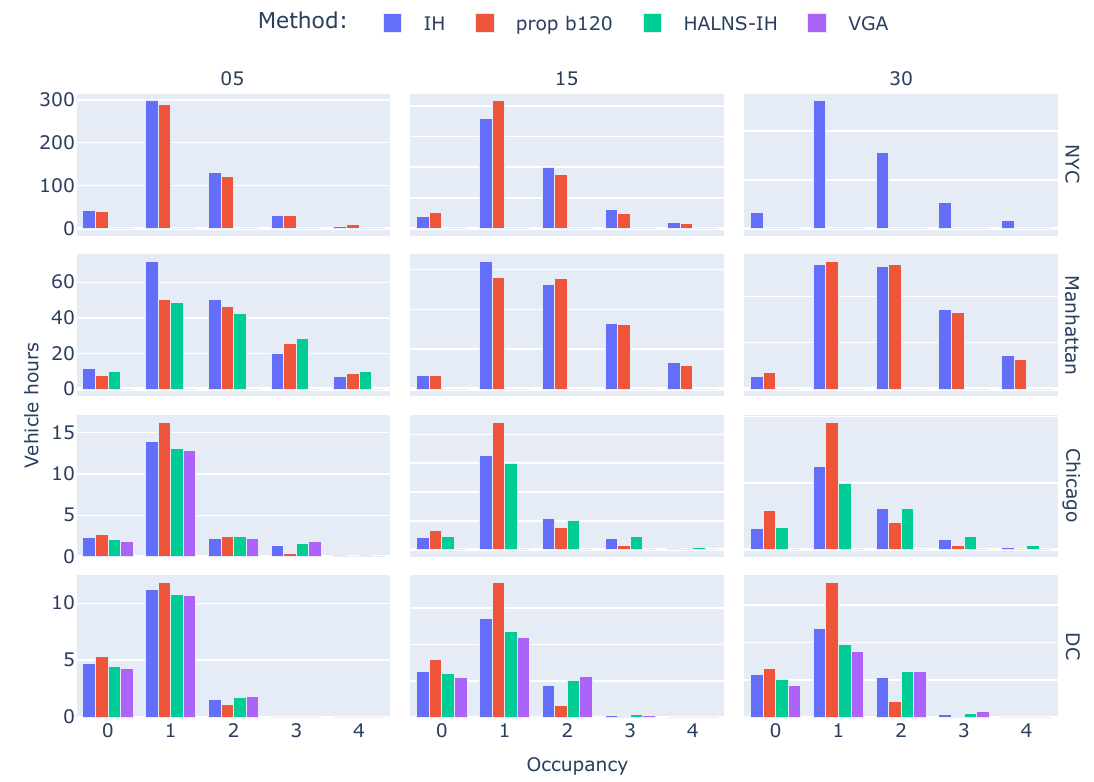}
    
    \caption{
    Occupancy histograms for all city-exp. length combinations.
    The \texttt{prop b120} method is the proposed method with a batch length of 120s.
    }
    \label{fig:occupancy}
\end{figure}
To make the histogram clear, we included only one variant of the proposed heuristic and one ALNS variant.
In general, we can see that the demand in the Chicago and DC areas is not dense enough to make significant savings by ridesharing. 
The vehicle hours with occupancy of two requests are compensated by empty vehicle hours and only a small fraction of vehicle hours are driven with more than two requests on board.
However, in the NYC area, the effect of ridesharing is already significant, and in Manhattan, we can see a high proportion of high occupancy, even fully occupied vehicles.
When comparing the occupancy between individual methods, we can see a correlation with the solution cost: the methods with lower solution costs tend to have higher occupancy.

We also examined delays for each individual travel request.
The histogram of the delays is in Figure~\ref{fig:delay}.
In general, the delays are higher for the methods that offer lower operational costs, as the increased share of shared trips causes more delays.
However, this relationship is not strict.
In the many histograms, we can observe that the IH and ALNS cause the highest delays, but the most efficient is the proposed method, or VGA (compare, e.g., to Table~\ref{tab:results_man}).
\begin{figure}
    \centering
    \includegraphics[width=1\textwidth]{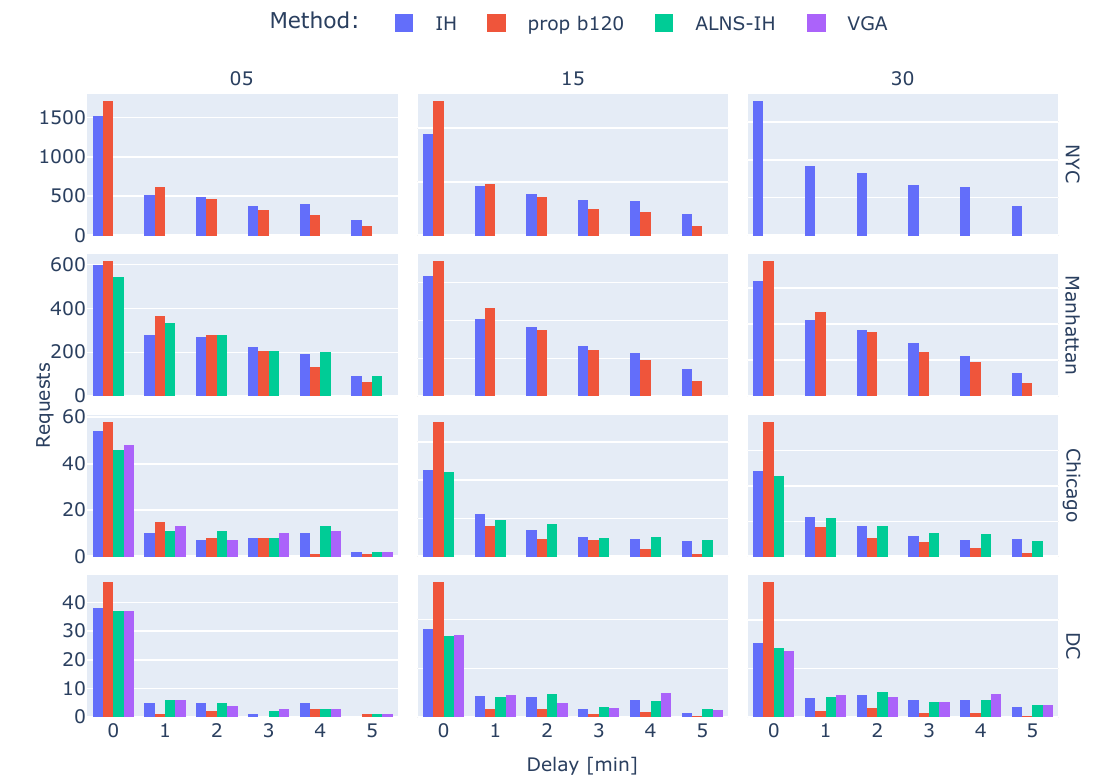}
    
    \caption{
    Delay histograms for all city-exp. length combinations.
    The \texttt{prop b120} method is the proposed method with a batch length of 120s.
    }
    \label{fig:delay}
\end{figure}

Overall, it is evident that we cannot choose a method that is best for all the instances.
Rather than that, one message of these results, also previously indicated in~\textcite{fiedlerLargescaleRidesharingDARP2023}, is that it is essential to test the DARP methods on multiple instances with different characteristics, before making any assumptions on the best method.
For small instances with low computational complexity, we can use the optimal method, even if the time horizon is long.
For the hardest instances, only the IH is applicable. 
Finally, for the remaining instances with medium complexity, the proposed method is usually the best.
Note also that the ALNS\nobreakdash-prop method, which sometimes outperforms all configurations of the proposed method, is initialized with the solution of the proposed method. 
Therefore, we think that this demonstration of the possible application of the optimal chaining proves its practical usability, and should be in the toolbox of future researchers focused on DARP.

\section{Conclusion}
For many problems related to on-demand mobility (MoD), it is desirable to connect vehicle plans into plan sequences, which can be ultimately seen as a new, longer, plan.
Some examples of such problems are fleet sizing or dial-a-ride problem (DARP).
There already exist methods for chaining (connecting) plans. 
However, they do not consider an important aspect of MoD: the time windows. 
Typically, each travel request has some time flexibility, called a time window.
These time windows propagate to the vehicle plans.
In conclusion, for optimal plan chaining, it is necessary to consider those time windows.
In this work, for the first time, we formulate the plan chaining problem with time windows.
Next, we propose a solution method for solving the chaining problem, significantly reducing the search space compared to naive solutions.
Moreover, we prove that the method is optimal.
Finally, we present a demonstration of the proposed chaining method that exhibits its potential to help solve large-scale DARPs. 
To solve the DARP, we split the demand by start time into batches, then solve each batch optimally, and finally, chain the plans using the proposed chaining method.
The demonstration results confirm the applicability of the proposed chaining method, as the heuristic based on it can solve instances that cannot be solved optimally in the time limit while being better than other evaluated heuristics in most of the evaluated instances.
In the future, we would like to evaluate the optimal chaining in other contexts: first, for instances with much longer duration (e.g., \SI{24}{h}) and second, for fleet-sizing.

\section*{Acknowledgments}
This work was supported by the Technology Agency of the Czech Republic within the DOPRAVA 2020+ program, project no. \texttt{CK04000150}.
Fabio V.~Difonzo gratefully thanks the INdAM-GNCS group for partial support. 
The access to the computational infrastructure of the OP VVV funded project \texttt{CZ.02.1.01/0.0/0.0/16\_019/0000765} "Research Center for Informatics" is also gratefully acknowledged.

\section*{Conflicts of Interest}
The authors declare no conflict of interest.

\printbibliography

\end{document}